\documentclass[a4paper,reqno,10pt]{amsart}

\usepackage{amsmath}
\usepackage{amsthm}
\usepackage{amssymb}
\usepackage{amscd} 

\usepackage{bbm} 
\usepackage{mathrsfs} 

\usepackage{lmodern} 
\usepackage[T1]{fontenc} 
\usepackage{newtxtext,newtxmath}

\usepackage{manfnt} 

\usepackage{graphicx} 

\usepackage{verbatim} 

\usepackage[Lenny]{fncychap} 

\usepackage{sqrcaps} 

\usepackage{comment}

\usepackage[svgnames]{xcolor}
\usepackage{pdfcolmk}

\usepackage{tikz-cd}
\usepackage{pifont}

\swapnumbers 

\newtheoremstyle{exercise} 
  {3pt} 
  {3pt} 
  {\small\rmfamily} 
  {
} 
  {\rmfamily\scshape} 
  {.} 
  {.5em} 
  {} 

\newtheoremstyle{newplain}
  {5pt}
  {5pt}
  {\itshape}
  {}
  {\rmfamily\scshape}
  {. ---}
  {.5em}
  {}

\newtheoremstyle{newremark}
  {5pt}
  {5pt}
  {\rmfamily}
  {}
  {\rmfamily\scshape}
  {. ---}
  {.5em}
  {}

\theoremstyle{newplain}
\newtheorem*{Theorem*}{Theorem} 

\theoremstyle{newplain}
\newtheorem{Theorem}{Theorem}
\newtheorem{Lemma}[Theorem]{Lemma}

\newtheorem{Proposition}[Theorem]{Proposition}
\newtheorem{Conjecture}[Theorem]{Conjecture}
\newtheorem{Definition}[Theorem]{Definition}

\theoremstyle{newremark}
\newtheorem{Empty}[Theorem]{}
\newtheorem{Remark}[Theorem]{Remark}

\newtheorem{Claim}[Theorem]{Claim}

\theoremstyle{exercise}

\numberwithin{Theorem}{section}
\numberwithin{Exercise}{section}

\newcommand{\R}{\mathbb{R}} 

\newcommand{\Rn}{\R^n}

\newcommand{\ind}{\mathbbm{1}} 

\newcommand{\calA}{\mathscr{A}}
\newcommand{\calB}{\mathscr{B}}
\newcommand{\calC}{\mathscr{C}}

\newcommand{\calE}{\mathscr{E}}

\newcommand{\calH}{\mathscr{H}}
\newcommand{\calI}{\mathscr{I}}

\newcommand{\calL}{\mathscr{L}}
\newcommand{\calM}{\mathscr{M}}

\newcommand{\balpha}{\boldsymbol{\alpha}}

\newcommand{\bG}{\mathbf{G}}

\DeclareMathOperator{\rmTan}{\mathrm{Tan}} 

\newcommand{\hel} {
\hskip2.5pt{\vrule height7pt width.5pt depth0pt}
\hskip-.2pt\vbox{\hrule height.5pt width7pt depth0pt}
\, }

\def\XXint#1#2#3{{%
\setbox0=\hbox{$#1{#2#3}{\int}$}
\vcenter{\hbox{$#2#3$}}\kern-.5\wd0}}

\newcommand{\vphi}{\varphi}

\newcommand{\la}{\langle}
\newcommand{\ra}{\rangle}

\renewcommand{\em}{\bf}

\renewcommand{\leq}{\leqslant}
\renewcommand{\geq}{\geqslant}
\renewcommand{\subset}{\subseteq}

\newlength{\drop}

\usepackage{trajan}

\newcommand{\defeq}{\mathrel{\mathop:}=}

\begin{document}



\title{A representation formula for members of SBV dual}

\def\curraddrname{{\itshape On leave of absence from}}

\author[Ph. Bouafia]{Philippe Bouafia}

\address{F\'ed\'eration de Math\'ematiques FR3487 \\
  CentraleSup\'elec \\
  3 rue Joliot Curie \\
  91190 Gif-sur-Yvette
}

\email{philippe.bouafia@centralesupelec.fr}

\author[Th. De Pauw]{Thierry De Pauw}

\address{School of Mathematical Sciences\\
Shanghai Key Laboratory of PMMP\\ 
East China Normal University\\
500 Dongchuang Road\\
Shanghai 200062\\
P.R. of China\\
and NYU-ECNU Institute of Mathematical Sciences at NYU Shanghai\\
3663 Zhongshan Road North\\
Shanghai 200062\\
China}
\curraddr{Universit\'e Paris Diderot\\ 
Sorbonne Universit\'e\\
CNRS\\ 
Institut de Math\'ematiques de Jussieu -- Paris Rive Gauche, IMJ-PRG\\
F-75013, Paris\\
France}
\email{thdepauw@math.ecnu.edu.cn,thierry.de-pauw@imj-prg.fr}

\keywords{Hausdorff measure, integral geometric measure, dual space, functions of bounded variation}

\subjclass[2010]{Primary 28A78,28A75,26A45,46E99}

\thanks{The first author was partially supported by the Science and Technology Commission of Shanghai (No. 18dz2271000).}



\begin{abstract}
We give an integral representation formula for members of the dual of $SBV(\Rn)$ in terms of functions that are defined on $\hat{\R}^n$, an appropriate fiber space that we introduce, consisting of pairs $(x,[E]_x)$ where $[E]_x$ is an approximate germ of an $(n-1)$-rectifiable set $E$ at $x$.
\end{abstract}

\maketitle

\section{Introduction}

Let $n \geq 2$ and $\Omega \subset \Rn$ be open.
The Banach spaces $BV(\Omega)$ and $SBV(\Omega)$ consist, respectively, of functions of bounded variation \cite[3.1]{AMBROSIO.FUSCO.PALLARA} and special functions of bounded variation \cite[\S 4.1]{AMBROSIO.FUSCO.PALLARA}.
A long-standing open problem \cite[7.4]{ARCATA} is to provide a useful description of the dual of $BV(\Omega)$. 
This seems to be still open, despite important contributions by N.G. Meyers and W.P. Ziemer \cite{MEY.ZIE.77}, N.C. Phuc and M. Torres \cite{PHU.TOR.17},  and N. Fusco and D. Spector \cite{FUS.SPE.18}.
Regarding the dual of $SBV(\Omega)$, the second author contributed the article \cite{DEP.98}.
In the present paper, we describe the dual of $SBV(\Omega)$ in a way
that is ``optimal'' in some specific universal sense.  \par For each
$u \in SBV(\Omega)$, the distributional gradient $Du$ of $u$
decomposes as $Du = \calL^n \hel \nabla u + \calH^{n-1} \hel j_u$, where
$\nabla u$ is the pointwise a.e. approximate gradient of $u$, and
$j_u$ is a vector field carried on the approximate discontinuity set
$S_u$ of $u$, on which we have $j_u = (u^+-u^-)\nu_u$, $\nu_u$ being a
unitary field normal to $S_u$, and $u^+, u^-$ the approximate limits
of $u$ on either sides of $S_u$.
We have
\begin{equation*}
  \|u\|_{SBV(\Omega)} = \int_\Omega |u| \, d \calL^n + \int_{S_u} \|j_u\| \,
  d \calH^{n-1} + \int_\Omega \|\nabla u\| \, d \calL^n.
\end{equation*}
Thus, $u \mapsto (u, j_u, \nabla u)$ is an isometric embedding
\begin{equation*}
  SBV(\Omega) \to L^1(\Omega, \calB(\Omega), \calL^n) \times
  L^1(\Omega, \calB(\Omega), \calH^{n-1}; \R^n) \times L^1(\Omega,
  \calB(\Omega), \calL^n ; \R^n).
\end{equation*}
Since $(\Omega,\calB(\Omega),\calL^n)$ is $\sigma$-finite, the dual of the corresponding $L^1$ space is, of course, the corresponding $L^\infty$ space. 
This remark does not apply to the middle measure space $(\Omega,\calB(\Omega),\calH^{n-1})$. 
An attempt at understanding the dual of the corresponding $L^1$ space leads to a study of the canonical map
\begin{equation*}
\Upsilon \colon L^\infty(\Omega,\calB(\Omega),\calH^{n-1}) \to L^1(\Omega,\calB(\Omega),\calH^{n-1})^*\,.
\end{equation*}
This relates to the Radon-Nikod\'ym Theorem.
In this respect, we note that:
\begin{itemize}
\item $\Upsilon$ is not surjective. If $\calB(\Omega)$ is replaced
  with $\calM_{\calH^{n-1}}(\Omega)$, whether $\Upsilon$ is surjective
  or not is undecidable in ZFC.
\item $\Upsilon$ is injective. If $\calB(\Omega)$ is replaced with
  $\calM_{\calH^{n-1}}(\Omega)$, $\Upsilon$ is not injective.
\end{itemize}
Here, $\calM_{\calH^{n-1}}(\Omega)$ is the $\sigma$-algebra of
$\calH^{n-1}$-measurable subsets of $\Omega$.
For a detailed treatment of these, see \cite{DEP.21}.
Our goal now is to stick with the axioms of ZFC.
\par 
The problem is with measurability.
Let $\vphi \in L^1(\Omega,\calB(\Omega),\calH^{n-1})$ and let $E \subset \Omega$ be Borel such that $\calH^{n-1}(E) < \infty$.
If $\iota_E \colon L^1(E,\calB(E),\calH^{n-1}) \to L^1(\Omega,\calB(\Omega),\calH^{n-1})$ is the obvious injection, then $\vphi \circ \iota_E$ belongs to the dual of $L^1(E,\calB(E),\calH^{n-1})$ and, by the classical Riesz Theorem, is represented by a Borel measurable function $g_E \colon E \to \R$, i.e. $(\vphi \circ \iota_E)(f) = \int_E fg_E d\calH^{n-1}$ whenever $f \in L^1(E,\calB(E),\calH^{n-1})$.
The family $\la g_E\ra_E$ is compatible in the sense that $\calH^{n-1}(E \cap E' \cap \{ g_E \neq g_{E'}\})=0$ for all $E, E'$.
There does not exist, in general, a {\it gluing} of this family, i.e. a Borel measurable function $g \colon \Omega \to \R$ such that $\calH^{n-1}(E \cap \{ g \neq g_E \}) = 0$ for all $E$, and whether an $\calH^{n-1}$-measurable gluing of $\la g_E \ra_E$ exists is undecidable.
\par 
In fact, a ``proper'' gluing of $\la g_E \ra_E$ at $x \in \Omega$ depends both on $x$ and $E$.
However, it will depend on $(x,E)$ only according to the behavior of $E$ near $x$. 
This prompts us to introduce a notion of approximate germ of $E$ at $x$ -- finer than the first order tangential behavior in case $E$ is, for instance, a smooth hypersurface.
Carrying out these ideas to describe the dual of $L^1(\Omega,\calB(\Omega),\calH^{n-1})$ presents technical difficulties that we avoid by considering a slightly different measure space.  
The jump set $S_u$ of a function of bounded variation being countably
rectifiable, one has $\calH^{n-1} \hel S_u = \calI^{n-1}_\infty \hel
S_u$, where $\calI^{n-1}_\infty$ is an integral geometric measure.
The application to studying $SBV(\Omega)^*$ is not affected either when we replace $\calI^{n-1}_\infty$ by its semi-finite version $\calI^{n-1}$. 
Both measures are described in \ref{sf}.
What we have gained is that each set $E \in \calB(\Omega)$ with $\calI^{n-1}(E) < \infty$ has the property that $\Theta^{n-1}(E,x)=1$ for $\calI^{n-1}$-almost every $x \in E$, according to the Structure Theorem.
One recognizes a density property.
This is an essential ingredient of our proof below and would not hold with $\calH^{n-1}$ in place of $\calI^{n-1}$.
\par 
Notice that a family $\la g_E \ra_E$ is, in fact, the same thing as a function $g(x,E)$ that depends on both $x$ and $E$, as in \cite{DEP.98}, but the compatibility of the family $\la g_E \ra_E$ considered above means that there is then some redundancy in the corresponding space of variables $(x,E)$.
The gist of the present paper is to consider a special quotient of the space of pairs $(x,E)$ and to equip it with a structure of a measure space $(\hat{\Omega},\hat{\calA},\hat{\calI}^{n-1})$, see \ref{sf} below.
One can of course wonder if the quotient presented here is ``optimal'', i.e. the smallest one for the purpose of representing the dual of $L^1$, or if we are still left with some redundancy in the variables.
The optimality of the measure space $(\hat{\Omega},\hat{\calA},\hat{\calI}^{n-1})$ is stated in \cite{BOU.DEP.21} and established there, in an appropriate categorical setting. 
Thus, the combination of \cite{BOU.DEP.21} and the present paper can
be considered a solution to describing the dual of $SBV(\Omega)$, see
Theorem~\ref{main}, as explicitly as one possibly can in ZFC.

\begin{Empty}[Notations]
  Our choice of notation is mostly compatible with that of \cite{GMT}
  and \cite{AMBROSIO.FUSCO.PALLARA}. As an exception, we let $\ominus$
  be the set theoretic symmetric difference.  In $\Rn$, the
  $(n-1)$-dimensional Hausdorff measure is $\calH^{n-1}$ and the
  Lebesgue measure is $\calL^n$. The $\sigma$-algebras consisting of
  Borel subsets and $\calH^{n-1}$-measurable subsets of $E \subset
  \Rn$ are respectively denoted by $\calB(E)$ and
  $\calM_{\calH^{n-1}}(E)$. The restriction symbol for measures is
  $\hel$ \cite[2.1.2]{GMT}. The lower and upper $(n-1)$-dimensional
  densities of $E \subset \Rn$ at $x \in \Rn$, denoted
  $\Theta_*^{n-1}(E,x)$ and $\Theta^{*\,n-1}(E,x)$, respectively, are
  defined in \cite[2.10.19]{GMT} (with $\mu = \calL^n \hel E$). Both
  $x \mapsto \Theta_*^{n-1}(E,x)$ and $x \mapsto \Theta^{*\,n-1}(E,x)$
  are Borel measurable. When they coincide at $x$, the common value is
  denoted $\Theta^{n-1}(E,x)$.
\end{Empty}

\section{Results}

\begin{Empty}
  \label{para1}
  A set $E \subset \R^n$ is called \textbf{countably
    $(n-1)$-rectifiable} whenever there are countably many Lipschitz
  maps $f_k \colon \R^{n-1} \to \R^n$ such that
  \[
  \calH^{n-1}\left( E \setminus \bigcup_{k = 0}^\infty f_k(\R^{n-1})
  \right) = 0.
  \]
  Also, $E$ is called \textbf{$(n-1)$-rectifiable} if it is
  $\calH^{n-1}$-measurable, of finite $\calH^{n-1}$-measure, and
  countably $(n-1)$-rectifiable.

  Throughout the paper, we fix an open subset $\Omega \subset
  \R^n$. For an $(n-1)$-rectifiable set $E \subset \Omega$, the
  following hold:
  \begin{enumerate}
  \item[(A)] $\Theta^{n-1}(E, x) = 1$ for $\calH^{n-1}$-almost all $x
    \in E$, \cite[3.2.19]{GMT};
  \item[(B)] $\Theta^{n-1}(E, x) = 0$ for $\calH^{n-1}$-almost all $x
    \in \Omega \setminus E$, \cite[2.10.19(4)]{GMT}. 
  \end{enumerate}
  The set $\tilde{E} = \{x \in \Omega : \Theta^{n-1}(E, x) = 1\}$
  coincides $\calH^{n-1}$-almost everywhere with $E$, i.e. $\calH^{n-1}(\tilde{E} \ominus E) = 0$. It satisfies:
  \begin{enumerate}
  \item[(C)] $\tilde{E}$ is a Borel measurable $(n-1)$-rectifiable
    subset of $\Omega$;
  \item[(D)] A point $x \in \Omega$ is in $\tilde{E}$ iff
    $\Theta^{n-1}(\tilde{E}, x) = 1$.
  \end{enumerate}
\end{Empty}

\begin{Empty}
\label{sf}
For $S \subset \Omega$, we let
  \[
  \calI^{n-1}(S) = \sup \left\{ \calH^{n-1}(S \cap E) : E \subset
  \Omega \text{ is an } (n-1)\text{-rectifiable set} \right\}.
  \]
  It is easily checked that $\calI^{n-1}$ is a Borel regular outer
  measure on $\Omega$. In fact, the measure space
  $(\Omega,\calB(\Omega),\calI^{n-1})$ is the semi-finite version, in
  the sense of \cite[213X(c)]{FREMLIN.II}, of
  $(\Omega,\calB(\Omega),\calI^{n-1}_\infty)$, where the integral
  geometric measure $\calI^{n-1}_\infty$ is defined in
  \cite[2.10.5(1)]{GMT}. The latter is a consequence of the Structure
  Theorem \cite[3.3.14]{GMT} and \cite[3.2.26]{GMT}. The measure
  $\calI^{n-1}$ is much more manageable than the Hausdorff measure
  $\calH^{n-1}$, as it ignores purely unrectifiable sets. We notice
  that $\calH^{n-1} \hel E = \calI^{n-1} \hel E$ for any countably
  $(n-1)$-rectifiable Borel set $E \subset \Omega$. To find an
  integral representation for the elements of $SBV(\Omega)^*$, it is
  crucial to describe the dual of $L^1(\Omega, \calB(\Omega),
  \calI^{n-1})$. We will do this in Theorem~\ref{dualL1}. Prior to
  that, we must construct a particular measure space $(\hat{\Omega},
  \hat{\calA}, \hat{\calI}^{n-1})$.
\end{Empty}

\begin{Empty}[The measure space $(\hat{\Omega}, \hat{\calA},
    \hat{\calI}^{n-1})$] 
\label{fs}    
    Let $\calE$ be the collection of sets that
  satisfy conditions~(C) and~(D) of \ref{para1}.  For each
  $x \in \Omega$, we let $\calE_x \defeq \{E \in \calE : x \in
  E\}$. We define an equivalence relation $\sim_x$ on $\calE_x$:
  \[
  E \sim_x E' \text{ iff } \Theta^{n-1}(E \ominus E', x) = 0 \text{
    iff } \Theta^{n-1}(E \cap E', x) = 1.
  \]
  The equivalence class $[E]_x$ under $\sim_x$ of some $E \in \calE_x$
  is called the {\em approximate germ of $E$ at $x$}. We will use
  repeatedly the following consequence of \ref{para1}(A):
  if $E, E' \in \calE$, then $[E]_x = [E']_x$ for $\calH^{n-1}$-almost
  all $x \in E \cap E'$.

  We let $\hat{\Omega}$ be the set consisting of the pairs $(x,
  [E]_x)$, where $x \in \Omega$ and $[E]_x$ is some approximate germ
  at $x$. This set should be thought of as a fiber space over $\Omega$,
  with respect to the obvious projection map $p \colon \hat{\Omega} \to \Omega$
  that sends each $(x, [E]_x)$ to $x$. Below, we equip $\hat{\Omega}$
  with a $\sigma$-algebra and a measure.

  For each $E \in \calE$, we let $\gamma_E \colon E \to \hat{\Omega}$
  be the map defined by $ \gamma_E(x) \defeq (x, [E]_x)$. We define
  $\hat{\calA}$ to be the finest $\sigma$-algebra on $\hat{\Omega}$ so
  that all maps $\gamma_E$ are $(\calM_{\calH^{n-1}}(E),
  \hat{\calA})$-measurable, that is
  \[
  \hat{\calA} \defeq \left\{ A \subset \hat{\Omega} : \gamma_E^{-1}(A)
  \in \calM_{\calH^{n-1}}(E) \text{ for all } E \in \calE\right\}.
  \]
  For all $E \in \calE$, we denote by $\mu_E$ the measure on
  $(\hat{\Omega}, \hat{\calA})$ defined by $\mu_E(A) \defeq
  \calH^{n-1}(\gamma_E^{-1}(A))$.  We let $\hat{\calI}^{n-1}$ be the
  least upper bound of the measures $\mu_E$, with $E$ running over
  $\calE$ (see e.g. \cite[1.68]{AMBROSIO.FUSCO.PALLARA}). The measure
  space $(\hat{\Omega}, \hat{\calA}, \hat{\calI}^{n-1})$ is complete:
  a set $N \subset \hat{\Omega}$ is negligible iff $\gamma_E^{-1}(N)$
  is $\calH^{n-1}$-negligible for all $E \in \calE$.

  From the definition of the $\sigma$-algebra $\hat{\calA}$, it is
  easy to show that a map $g \colon \hat{\Omega} \to Z$ with values in
  a measurable space $(Z, \calC)$ is $(\hat{\calA}, \calC)$-measurable
  iff all maps $g \circ \gamma_E$ are $(\calM_{\calH^{n-1}}(E),
  \calC)$-measurable. As a consequence, the projection map $p$ is
  $(\hat{\calA}, \calB(\Omega))$-measurable, for the maps $p \circ
  \gamma_E$ are the inclusion maps $E \to \Omega$ and, therefore, are
  $\calH^{n-1}$-measurable.
\end{Empty}

\begin{Lemma}
  \label{propMeasI}
  The following hold:
  \begin{enumerate}
  \item For all $A \in \hat{\calA}$, one has $\hat{\calI}^{n-1}(A) =
    \sup \left\{ \mu_E(A) : E \in \calE
    \right\}$.
   
   
  \item $p_\# \hat{\calI}^{n-1} = \calI^{n-1}$.
  \item For $E \in \calE$, we have $\hat{\calI}^{n-1}(p^{-1}(E)
    \setminus \gamma_E(E)) = 0$
  \end{enumerate}
\end{Lemma}

\begin{proof}
  (1) Letting $\nu(A)$ denote the right hand side of the claimed equality, it suffices to check that $A \mapsto \nu(A)$ defines a
  measure. Clearly, $\nu$ is
  non-decreasing and $\sigma$-subadditive.

  It remains to check that $\nu(A_1 \cup A_2) \geq
  \nu(A_1) + \nu(A_2)$ for any pair of
  disjoint $A_1, A_2 \in \hat{\calA}$. Let $E_1, E_2 \in
  \calE$. Recalling the discussion in \ref{para1}, there is
  a set $E \in \calE$ that coincides $\calH^{n-1}$-almost everywhere
  with $E_1 \cup E_2$. For $i \in \{1, 2\}$, we have $[E]_x = [E_i]_x$
  for $\calH^{n-1}$-almost all $x \in E \cap E_i$. From this, it
  follows that
  \[
  \gamma_{E_i}^{-1}(A_i) \setminus \gamma_E^{-1}(A_i) \subset (E_i
  \setminus E) \cup \{x \in E \cap E_i : [E]_x \neq [E_i]_x\}
  \]
  is $\calH^{n-1}$-negligible. Hence, $\mu_E(A_i) \geq \mu_{E_i}(A_i)$.
  As $A_1$ and $A_2$ are disjoint, one has $\mu_E(A_1 \cup A_2) =
  \mu_E(A_1) + \mu_E(A_2) \geq \mu_{E_1}(A_1) + \mu_{E_2}(A_2)$.  The
  proof is then completed by taking the supremum over $E_1, E_2 \in
  \calE$.

  (2) For all $B \in \calB(\Omega)$, one has by~(1)
  \begin{align*}
  \hat{\calI}^{n-1}(p^{-1}(B)) & = \sup \left\{
  \calH^{n-1}(\gamma_E^{-1}p^{-1}(B)) : E \in \calE \right\} \\ & =
  \sup \left\{ \calH^{n-1}(B \cap E) : E \in \calE \right\},
  \end{align*}
  which is $\calI^{n-1}(B)$.

  (3) Let $E' \in \calE$. The set
  \[
  \gamma_{E'}^{-1}(p^{-1}(E) \setminus \gamma_E(E)) = \{x \in E \cap
  E' : [E]_x \neq [E']_x \}
  \]
  has $\calH^{n-1}$ measure zero. By the arbitrariness of $E'$, it
  follows that $p^{-1}(E) \setminus \gamma_E(E) \in \hat{\calA}$ and
  $\hat{\calI}^{n-1}(p^{-1}(E) \setminus \gamma_E(E)) = 0$.
\end{proof}

\begin{Empty}[Density points]
  We recall the following well-known facts regarding the existence of
  Lebesgue density points for functions defined on rectifiable
  sets. Let $E$ be an $(n-1)$-rectifiable set and $h \colon E \to \R$
  be a function in $L^1(E, \calM_{\calH^{n-1}}(E), \calH^{n-1} \hel
  E)$, we have
  \[
  \frac{1}{\balpha(n-1)r^{n-1}} \int_{E \cap B(x,r)} h \, d
  \calH^{n-1} \xrightarrow[r \to 0]{} h(x)
  \]
  at $\calH^{n-1}$-almost every $x \in E$. When the limit exists, $x$
  is called a density point of $h$. The proof of the next lemma is
  elementary.
\end{Empty}

\begin{Lemma}
  \label{sameDens}
  Let $E, E'$ be two Borel measurable $(n-1)$-rectifiable susbsets of
  $\Omega$, $h \in L^\infty(E, \calM_{\calH^{n-1}}(E), \calH^{n-1}
  \hel E)$, $h' \in L^\infty(E', \calM_{\calH^{n-1}}(E'), \calH^{n-1}
  \hel E')$ be two functions, and $x \in E \cap E'$. Assume that
  \begin{enumerate}
  \item[(A)] $h = h'$ $\calH^{n-1}$-almost everywhere on $E \cap E'$;
  \item[(B)] $\Theta^{n-1}(E,x) = \Theta^{n-1}(E', x) = \Theta^{n-1}(E \cap E', x) = 1$;
  \item[(C)] $x$ is a density point of $h$.
  \end{enumerate}
  Then $x$ is a density point of $h'$ and
  \[
  \lim_{r \to 0} \frac{1}{\balpha(n-1)r^{n-1}} \int_{E \cap B(x,r)} h
  \, d \calH^{n-1} = \lim_{r \to 0} \frac{1}{\balpha(n-1)r^{n-1}}
  \int_{E' \cap B(x,r)} h' \, d \calH^{n-1}.
  \]
\end{Lemma}

We now come to our main result.

\begin{Theorem}
  \label{dualL1}
  The map $\Upsilon \colon L^\infty(\hat{\Omega}, \hat{\calA},
  \hat{\calI}^{n-1}) \to L^1(\Omega, \calB(\Omega), \calI^{n-1})^*$
  defined by
  \[
  \Upsilon(g)(f) = \int_{\hat{\Omega}} g(f \circ p) \, d\hat{\calI}^{n-1}
  \]
  is an isometric isomorphism.
\end{Theorem}

\begin{proof}
  Let $g \in L^\infty(\hat{\Omega}, \hat{\calA}, \hat{\calI}^{n-1})$
  and $f \in L^1(\Omega, \calB(\Omega), \calI^{n-1})$. Then
  \begin{equation*}
    |\Upsilon(g)(f)|  \leq \|g\|_\infty \int_{\hat{\Omega}} |f \circ p| \, d \hat{\calI}^{n-1} 
     = \|g\|_\infty \int_{\hat{\Omega}} |f| \, dp_{\#}\hat{\calI}^{n-1} = \|g\|_\infty\|f\|_1,
  \end{equation*}
  by Lemma~\ref{propMeasI}(2). This shows that $\Upsilon$ is
  well-defined and $\|\Upsilon\| \leq 1$.

  Let $\varepsilon > 0$. The set $A \defeq \{|g| > \|g\|_\infty -
  \varepsilon\}$ has positive $\hat{\calI}^{n-1}$ measure, hence we
  infer the existence of an $(n-1)$-rectifiable set $E \in \calE$ such
  that $\calH^{n-1}(\gamma_E^{-1}(A)) > 0$. The function $g \circ
  \gamma_E$ is $\calH^{n-1}$-measurable. As the measure space $(E,
  \calM_{\calH^{n-1}}(E), \calH^{n-1} \hel E)$ is the completion of
  $(E, \calB(E), \calH^{n-1} \hel E)$, we can find a Borel measurable
  function $g_E \colon E \to \R$ that coincides almost everywhere with
  $g \circ \gamma_E$. We let $f \in L^1(\Omega, \calB(\Omega),
  \calI^{n-1})$ be the map
  \[
  f \colon x \mapsto \begin{cases} 1 & \text{if } x \in E \text{ and }
    g_E(x) > \|g\|_\infty - \varepsilon \\ -1 & \text{if } x \in E
    \text{ and } g_E(x) < -\left(\|g\|_\infty - \varepsilon \right)
    \\ 0 & \text{otherwise}
  \end{cases}
  \]
  of norm $\|f\|_1 = \calH^{n-1}(\gamma_E^{-1}(A))$. As $p \circ
  \gamma_E$ is the inclusion map $E \to \Omega$, we have
  $\gamma_E(p(z)) = z$ for all $z \in \gamma_E(E)$. By
  Lemma~\ref{propMeasI}(3), it then follows that $\gamma_E(p(z)) = z$
  for $\hat{\calI}^{n-1}$-almost all $z \in p^{-1}(E)$. Hence, $g$ and
  $g \circ \gamma_E \circ p$ coincide $\hat{\calI}^{n-1}$-everywhere
  on $p^{-1}(E)$. Moreover, we have
  \[
  p^{-1}(E) \cap \{ g \circ \gamma_E \circ p \neq g_E \circ p\} \subset p^{-1}(\{g
  \circ \gamma_E \neq g_E\})
  \]
  By the Borel regularity of $\hat{\calI}^{n-1}$ and
  Lemma~\ref{propMeasI}(2), we also infer that $g \circ \gamma_E \circ
  p =g_E \circ p$ almost everywhere on $p^{-1}(E)$. Whence  
  \begin{align*}
  \Upsilon(g)(f) = \int_{p^{-1}(E)} g(f \circ p) \, d
  \hat{\calI}^{n-1} & = \int_{p^{-1}(E)} (g_E \circ p) (f \circ p) \,
  d \hat{\calI}^{n-1} \\ & = \int_E g_E f \, d
  \calI^{n-1},
  \end{align*}
  Thus, by construction of $f$, we have $\Upsilon(g)(f) \geq
  (\|g\|_\infty - \varepsilon) \|f\|_1$. This shows, as $\varepsilon >
  0$ can be taken arbitrarily small, that $\|\Upsilon(g)\| = 1$ and
  $\Upsilon$ is an isometry.

  Now we turn to establishing the surjectivity of $\Upsilon$. For any $E \in
  \calE$, we define a map $\alpha_E \in L^1(E, \calB(E), \calH^{n-1}
  \hel E)^*$ that is the localized version of $\alpha$: for an
  integrable function $f \in L^1(E, \calB(E), \calH^{n-1} \hel E)$, we
  set $\alpha_E(f) \defeq \alpha(\overline{f})$, where $\overline{f}$
  is the extension by zero of $f$ to $\Omega$. Since $(E, \calB(E),
  \calH^{n-1} \hel E)$ is a finite measure space, the standard duality
  between $L^1$ and $L^\infty$ spaces holds, so there is a function
  $g_E \in L^\infty(E, \calB(E), \calH^{n-1} \hel E)$ such that
  \[
  \alpha_E(f) = \int_E g_E f \, d\calH^{n-1} \text{ for all } f \in
  L^1(E, \calB(E), \calH^{n-1} \hel E).
  \]
  Obviously, for any $E, E' \in \calE$, we have
  \[
  \int_{E\cap E'} g_E f \, d \calH^{n-1} = \int_{E \cap E'} g_{E'} f
  \, d \calH^{n-1}
  \]
  for any integrable function $f$ on $E \cap E'$, which implies that
  $g_E = g_{E'}$ almost everywhere on $E \cap E'$.

  Let $g \colon \hat{\Omega} \to \R$ be the function partially defined
  by
  \[
  g(x, [E]_x) \defeq \lim_{r \to 0} \frac{1}{\balpha(n-1) r^{n-1}}
  \int_{E \cap B(x,r)} g_E \, d \calH^{n-1}
  \]
  whenever $x$ is a density point of $g_E$. The compatibility between
  the local Radon-Nikod\'{y}m derivatives $g_E$, $E \in \calE$,
  together with Lemma~\ref{sameDens}, guarantees that $g$ is
  well-defined. Denoting by $N \subset \hat{\Omega}$ the set of points
  at which $g$ is not defined, we readily see that, for any $E \in
  \calE$, the set $\gamma_E^{-1}(N)$ consists of the $x \in E$ that
  are not density points of $g_E$. We readily have
  $\calH^{n-1}(\gamma_E^{-1}(N)) = 0$, and this shows that
  $\hat{\calI}^{n-1}(N) = 0$. We can extend $g$ to $\hat{\Omega}$ by
  sending the elements of $N$ to an arbitrary value. Then $g$ is
  $\hat{\calA}$-measurable. Indeed, for any $E \in \calE$ the function
  $g \circ \gamma_E$ coincides $\calH^{n-1}$-almost every with $g_E$,
  and therefore is $\calH^{n-1}$-measurable.

  Finally, we claim that $\alpha = \Upsilon(g)$. Let $f \in L^1(\Omega,
  \calB(\Omega), \calI^{n-1})$. Choose $\varepsilon > 0$ and set $E
  \defeq \{|f| > \varepsilon\}$. As $\calI^{n-1}(E) < \infty$, we can
  suppose, up to a modification of $f$ on a negligible set, that $E
  \in \calE$. Let $f_E$ be the restriction of $f$ to $E$. Then,
  recalling that $\calH^{n-1} \hel E = \calI^{n-1} \hel E$, we have
  \[
  \alpha(\ind_{\{|f| > \varepsilon\}} f) = \alpha_E(f_E) = \int_E g_E f_E \,
  d\calI^{n-1}.
  \]
  By Lemma~\ref{propMeasI}(2), we deduce that
  \[
  \alpha(\ind_{\{|f| > \varepsilon\}} f) = \int_{p^{-1}(E)} (g_E \circ
  p)(f_E \circ p) \, d \hat{\calI}^{n-1}.
  \]
  Clearly, $f_E \circ p = f \circ p$ on $p^{-1}(E)$ and $g_E\circ p =
  g$ almost everywhere on $\gamma_E(E)$. By Lemma~\ref{propMeasI}(3),
  this ensures that $g_E \circ p = g \circ p$ almost everywhere on
  $p^{-1}(E)$. Thus,
  \[
  \alpha(\ind_{\{|f| > \varepsilon\}} f) = \int_{p^{-1}(E)} g (f \circ p)
  \, d \hat{\calI}^{n-1} = \int \ind_{\{|f \circ p| > \varepsilon\}} g (f
  \circ p) \, d \hat{\calI}^{n-1}.
  \]
  Letting $\varepsilon \to 0$ yields $\alpha(f) = \Upsilon(g)(f)$.
\end{proof}

\begin{Remark}
  The above proof exhibits a construction of the density $g$, by
  gluing the local Radon-Nikod\'ym derivatives $g_E$ in the measure
  space $(\hat{\Omega}, \hat{\calA}, \hat{\calI}^{n-1})$. The gluing
  is not possible within the measure space $(\Omega, \calB(\Omega),
  \calI^{n-1})$, as it is not localizable (if $\Omega$ is non
  empty). It was proven in \cite[Section 11]{BOU.DEP.21} that
  $(\hat{\Omega}, \hat{\calA}, \hat{\calI}^{n-1})$ is the ``minimal''
  measure space in which a compatible family of functions defined on
  subsets of $\Omega$ can be glued. Using the concept introduced in
  \cite{BOU.DEP.21}, $(\hat{\Omega}, \hat{\calA}, \hat{\calI}^{n-1})$
  is the strictly localizable version of the completion of $(\Omega,
  \calB(\Omega), \calI^{n-1})$.

  Assuming that $(\hat{\Omega}, \hat{\calA}, \hat{\calI}^{n-1})$ is
  localizable, we can give an alternate proof of Theorem~\ref{dualL1}:
  the map $\iota \colon L^1(\Omega, \calB(\Omega), \calI^{n-1}) \to
  L^1(\hat{\Omega}, \hat{\calA}, \hat{\calI}^{n-1})$ that sends $f$ to
  $f \circ p$ is an isometry, as $\|f \circ p\|_1 = \|f\|_1$ by
  Lemma~\ref{propMeasI}(3). By the localizability of $(\hat{\Omega},
  \hat{\calA}, \hat{\calI}^{n-1})$, the dual of $L^1(\hat{\Omega},
  \hat{\calA}, \hat{\calI}^{n-1})$ is $L^\infty(\hat{\Omega},
  \hat{\calA}, \hat{\calI}^{n-1})$, in such a way that $\Upsilon$ is
  just the adjoint map of $\iota$. Reasoning as before, $\Upsilon$ is
  injective, which entails that $\iota$ and then $\Upsilon$ are
  isometric isomorphisms.
\end{Remark}

\begin{Empty}[Approximate tangent cone of an approximate germ]
  \label{para2.9}
  Let $E \in \calE$ and $x \in E$. The approximate tangent cone
  $\rmTan^{n-1}(E, x)$ (see \cite[3.2.16]{GMT}) is unchanged if we
  substitute $E$ with any $E' \in \calE_x$ such that $E \sim_x
  E'$. Because of this, we can define the approximate tangent cone
  $\rmTan^{n-1}([E]_x) \defeq \rmTan^{n-1}(E, x)$ of an approximate
  germ. We recall that $\rmTan^{n-1}(E, x)$ is an $(n-1)$-plane at
  $\calH^{n-1}$-almost every point $x \in E$, and that the function $x
  \mapsto \rmTan^{n-1}(E, x) \in \bG(n-1, n)$, defined almost
  everywhere on $E$, is $\calH^{n-1}$-measurable. Here $\bG(n-1, n)$
  denotes the Grassmann manifold of hyperplanes in $\R^n$.
\end{Empty}

\begin{Theorem}
\label{main}
  Let $\varphi \colon SBV(\Omega) \to \R$ be a continuous linear
  functional. There are fields $f \in L^\infty(\Omega, \calB(\Omega),
  \calL^n)$, $g \in L^\infty(\hat{\Omega}, \hat{\calA},
  \hat{\calI}^{n-1} ; \R^n)$ and $h \in L^\infty(\Omega,
  \calB(\Omega), \calL^n ; \R^n)$ such that
  \begin{equation}
    \label{dualSBV}
    \varphi(u) = \int_\Omega fu \, d \calL^n + \int_{\hat{\Omega}} g \cdot (j_u \circ
    p) d \hat{\calI}^{n-1} + \int_\Omega h \cdot \nabla u \, d \calL^n
  \end{equation}
  where the distributional gradient $Du$ of $u$ is decomposed into its
  Lebesgue part $\calL^n \hel \nabla u$ and its jump part $D^j u = j_u
  \, \calH^{n-1} \hel S_u$. Moreover, one can select $g$ in such a way
  that $g(x, [E]_x)$ is normal to $\rmTan^{n-1}([E]_x)$ for
  $\hat{\calI}^{n-1}$-almost all $(x, [E]_x) \in \hat{\Omega}$.
\end{Theorem}

\begin{proof}
  For each $u \in SBV(\Omega)$, we have
  \[
  \|u\|_{SBV(\Omega)} = \int_\Omega |u| \, d \calL^n + \int_{S_u} \|j_u\| \,
  d \calH^{n-1} + \int_\Omega \|\nabla u\| \, d \calL^n.
  \]
  Here, the approximate discontinuity set $S_u$ (see
  \cite[3.63]{AMBROSIO.FUSCO.PALLARA}) is Borel measurable and
  countably $(n-1)$-rectifiable \cite[3.64(a) and
    3.78]{AMBROSIO.FUSCO.PALLARA}. We recall that $j_u = 0$ outside
  $S_u$, whereas $j_u = (u^+ - u^-) \nu_u$ almost everywhere on
  $S_u$. In particular, $\calH^{n-1} \hel j_u =  \calI^{n-1} \hel j_u$. Therefore, $u \mapsto (u, j_u, \nabla u)$ is an
  isometric embedding
  \[
  SBV(\Omega) \to L^1(\Omega, \calB(\Omega), \calL^n) \times
  L^1(\Omega, \calB(\Omega), \calI^{n-1}; \R^n) \times L^1(\Omega,
  \calB(\Omega), \calL^n ; \R^n).
  \]
  Hence, $\varphi$ can be split into three parts $\varphi \colon u
  \mapsto \varphi_1(u) + \varphi_2(j_u) + \varphi_3(\nabla u)$, where
  $\varphi_1 \in L^1(\Omega, \calB(\Omega), \calL^n)^*$, $\varphi_2
  \in L^1(\Omega, \calB(\Omega), \calI^{n-1}; \R^n)^*$, and $\varphi_3
  \in L^1(\Omega,\calB(\Omega), \calL^{n-1}; \R^n)^*$. The
  duality theorem~\ref{dualL1} provides an integral representation for the term
  $\varphi_2$ that acts on the jump part of the derivative and yields
  the formula~\eqref{dualSBV}.

  Regarding the second assertion, we claim that $j_u \circ p(x, [E]_x)
  = j_u(x)$ is normal to $\rmTan^{n-1}([E]_x)$ for almost all $(x,
  [E]_x) \in \hat{\Omega}$. For this, we need to show that, for all $E
  \in \calE$,
  \begin{multline*}
    \{x \in E : j_u(x) \text{ not normal to } \rmTan^{n-1}(E, x)\}
    \\ = \{x \in E \cap S_u : j_u(x) \text{ not normal to }
    \rmTan^{n-1}(E, x)\}
  \end{multline*}
  is $\calH^{n-1}$-negligible. This is clear, once we remark that
  $\rmTan^{n-1}(E, x) = \rmTan^{n-1}(S_u, x)$ for almost every $x \in
  E \cap S_u$ and that $j_u$ is normal to $S_u$ almost everywhere on
  $S_u$. Hence, we can replace the vector field $g$ by its normal part
  $\tilde{g}$ (that is, $\tilde{g}(x, [E]_x)$ is the orthogonal
  projection of $g(x, [E]_x)$ onto $\rmTan^{n-1}([E]_x)^\perp$). The
  measurability of $\tilde{g}$ follows from the discussion in
  Paragraph~\ref{para2.9}.
\end{proof}


\bibliographystyle{amsplain}
\bibliography{thdp}

\end{document}